\documentclass[11pt,reqno,draft]{amsart}
\usepackage{latexsym,amssymb}

\newtheorem{defi}{Definition}[section]
\newtheorem{proposition}[defi]{Proposition}

\newtheorem{rem}[defi]{Remark}
\newtheorem{lemma}[defi]{Lemma}
\newtheorem{theorem}[defi]{Theorem}

\newcommand{\defeq}{\mathrel{\mathrm{\raise0.1ex\hbox{:}\hbox{=}\strut}}}

\def\R{\mathbb R}
\def\NN{\mathbb{N}}
\def\E{\mathbb E}

\def\ee{\mathrm{e}}
\def\PP{\mathbb P}

\def\ve{\varepsilon}
\def\g{\widetilde{g}}

\def\la{\langle}
\def\ra{\rangle}

\def\A{A_{\lambda}}
\newcommand{\Norm}[1] {|\!|\!|#1|\!|\!|}
\begin{document}

\numberwithin{equation}{section}

\title[Invariant measures for SFDE's with superlinear drift terms ]{Invariant measures for stochastic functional differential equations with superlinear drift term}
\author[A. Es--Sarhir]{Abdelhadi Es--Sarhir$^{\ast\flat}$}
\author[O. van Gaans]{Onno van Gaans$^\ast$}
\author[M. Scheutzow]{Michael Scheutzow $^\flat$}
\address{Technische Universit\"at Berlin, Fakult\"at II, Institut f\"ur Mathematik, Sekr. Ma 7-5\newline Stra{\ss}e des 17. Juni 136, D-10623 Berlin, Germany}
\email{essarhir@math.tu-berlin.de} \email{ms@math.tu-berlin.de}
\address{Universiteit Leiden, Mathematisch Instituut, \newline Postbus 9512, 2300 RA Leiden, The Netherlands}

\email{vangaans@math.leidenuniv.nl}
\thanks{$^\star$Es--Sarhir and van Gaans acknowledge the support by a `VIDI subsidie' (639.032.510) of the Netherlands Organisation for Scientific Research (NWO)}
\thanks{$^\flat$Es--Sarhir and Scheutzow acknowledge support
from the DFG Forschergruppe 718 "Analysis and Stochastics in Complex
Physical Systems".}



\begin{abstract} We consider a stochastic functional differential
equation with an arbitrary Lipschitz diffusion coefficient depending
on the past. The drift part contains a term with superlinear growth
and satisfying a dissipativity condition. We prove tightness and
Feller property of the segment process to show existence of an
invariant measure.
\end{abstract}

\maketitle
\section{Introduction and preliminaries}
\noindent There have been quite some investigations on stationary
solutions of stochastic functional differential equations with
nonlinear diffusion coefficients, see for instance
\cite{bakhtin,chow,scheutzow} and references therein. One approach
is to rewrite the functional differential equation as a semilinear
infinite dimensional equation and use results on invariant measures
of such equations (see \cite{DaZa:96}). The operator induced by the
linear part of a functional differential equation is often not
dissipative. For results on invariant measures for non-dissipative
systems, see \cite{bonaccorsi,vangaans}. These results require that
the linear part is exponentially stable and that the Lipschitz
constant of the diffusion is small with respect to the decay of the
linear part. By means of a finite dimensional analysis it has been
shown that the Lipschitz constant of the diffusion coefficient may
be arbitrary large, provided the diffusion coefficient is uniformly
bounded (see \cite{reiss}).

\noindent In this paper we prove existence of an invariant measure
for stochastic functional differential equations with no boundedness
conditions on the diffusion coefficient nor conditions on the size
of its Lipschitz constant. Instead, we consider a stabilizing
feedback term in the drift with superlinear growth. Let $r>0$ and
denote by  $C([-r,0],\mathbb{R}^d)$ the space of $\mathbb{R}^d$ valued
continuous functions on $[-r,0]$ and let $g\colon
C([-r,0],\R^d)\to\mathbb{R}^d$ and $h\colon
C([-r,0],\R^{d}) \to \mathbb{R}^{d\times m}$ be Lipschitz functions with respect
to the maximum norm. Let
$(B(t))_{t\ge 0}$ denote a standard $\R^m$-valued Brownian motion
defined on a filtered probability space
$(\Omega,\mathcal{F},(\mathcal{F}_t)_t,\mathbb{P})$. We will show
existence of an invariant measure for the functional differential
equation
\begin{equation}\label{sdde00} \mathrm{d}x(t)=\Big( -x(t)\cdot|x(t)|^s +g(x_t)\Big)\mathrm{d}t + h(x_t)\mathrm{d}B(t),\quad t\ge 0,\end{equation}
where $s>0$ and $x_t$ denotes the segment of $x$ given by
\[x_t(\theta)=x(t+\theta),\quad \theta\in [-r,0].\]

\noindent In order to show existence of an invariant measure, we
consider the segments of a solution. In contrast to the scalar
solution process, the process of segments is a Markov process. We
show that the process of segments is also Feller and that there
exists a solution of which the segments are tight. Then we apply the
Krylov-Bogoliubov method.

\noindent Since the segment process has values in the infinite
dimensional space $C([-r,0],\R^d)$, boundedness in probability does
not automatically imply tightness. For solution processes of
infinite dimensional equations, one often uses compactness of the
orbits of the underlying deterministic equation to obtain tightness.
For an infinite dimensional formulation of the functional
differential equation, however, such a compactness property does not
hold.

\noindent Our proof of tightness involves a Lyapunov function
technique to obtain boundedness in probability for the segment
process $(x_t)_{t\geq 0}$. Further we use the assumption on the
coefficients for the deterministic part, and Kolmogorov's criterion
for the noise part. By using a monotonicity argument we prove the
Feller property for $(x_t)_{t\geq 0}$ which implies the existence of
an invariant measure by the Krylov-Bogoliubov Theorem. Our analysis
holds true for the more general equation
\begin{equation}\label{sde0}
 \left\{
\begin{array}{ll}
dx(t)=\Big(f(x(t))+g(x_t)\Big)dt+h(x_t)dB(t),\quad\mbox{for}\:\:
t\geq0,\\
x(s)=\varphi(s)\:\:\mbox{for}\:\: s\in[-r,0],
\end{array}
\right.
\end{equation}
where we assume the following hypotheses:
\begin{enumerate}

\item [${\bf (H_0)}$] $f:\:\R^d\rightarrow\R^d$ is continuous and
$$
\lim\limits_{|v|\to +\infty} \frac{\la f(v),v\ra}{|v|^2}=-\infty.
$$
\item [${\bf (H_1)}$]  $g:\:C([-r,0],\mathbb{R}^d)\rightarrow\R^d$,  $h:\:C([-r,0],\mathbb{R}^d)\rightarrow\R^{d\times m}$
are continuous and bounded on bounded subsets of $C([-r,0],\R^d)$.

\item [${\bf (H_2)}$] There exists a positive constant
$L$ such that for all $x$, $y\in C([-r,0],\mathbb{R}^d) $
\begin{equation*}
\begin{split}
\Big(2\la f(x(0))-f(y(0)),x(0)-y(0)\ra^{+}&+2\la
g(x)-g(y),x(0)-y(0)\ra^{+}\\&+\Norm{ h(x)-h(y)}^2\Big)\leq L \|x-y\|^2,
\end{split}
\end{equation*}
where $\Norm M :=(\mathrm{Tr} (M M^*))^{1/2}$ denotes the trace norm of the matrix $M$.
\end{enumerate}
\noindent  The initial process $\varphi$ has almost surely
continuous paths and is independent of $(B(t))_{t\geq 0}$ with
$\E\|\varphi(\cdot,\omega)\|^p<\infty$ for all $p\geq 2$.

\noindent Note that under hypotheses ${\bf (H_0)} $, ${\bf (H_1)}$
and ${\bf (H_2)} $ and thanks to \cite[Theorem 2.3]{RS}, equation
\eqref{sde0} has a unique global solution given by

$$
x(t)=x(0)+\int_0^tf(x(s))\:ds+\int_0^tg(x_s)\:ds+\int_0^th(x_s)\:dB(s)\quad
\mbox{for any $t>0$}.
$$

\noindent We will prove existence of an invariant measure $\mu$ for
the segment process $(x_t)_{t\geq0}$ associated to the solution
$x(t)_{t\geq 0}$. Of course our hypotheses ${\bf(H_1)}$ and
${\bf(H_2)}$ allow the coefficient $h$ to be degenerate which can
not guarantee uniqueness of $\mu$. For recent results on the
uniqueness of invariant measures for stochastic functional
differential equations, see \cite{HMS}.

\noindent We end this introduction by the
following elementary remark which is useful for our arguments in the
sequel of this paper.

\begin{rem}
Let $T>0$. Consider a stochastic process $x(t)$, $-r \le t \le T$ with continuous paths
and let $x_t$, $t\geq0$ be its associated segment process on
$[-r,0]$. If $x_0=\varphi$ and $p \ge 1$, then
\begin{equation*}
\E\sup\limits_{0\leq t\leq T}\|x_t\|^p\leq
\E\|\varphi\|^p+\E\sup\limits_{0\leq t\leq T}|x(t)|^p
\end{equation*}
\end{rem}
\begin{proof}
We have
\begin{equation*}
\begin{split}
\E\sup\limits_{0\leq t\leq T}\|x_t\|^p&=\E\sup\limits_{0\leq t\leq
T}\sup\limits_{-r\leq s\leq 0}|x(t+s)|^p\\
&=\E\sup\limits_{0\leq t\leq T}\sup\limits_{t-r\leq s\leq
t}|x(s)|^p\\
&=\E\sup\limits_{-r\leq s\leq T}|x(s)|^p\leq
\E\|\varphi\|^p+\E\sup\limits_{0\leq s\leq T}|x(s)|^p.
\end{split}
\end{equation*}
\end{proof}

\section{Tightness of the segment process $(x_t)_{t\geq 0}$ }
\noindent In this section we will prove tightness of the family
$\{x_t:\, t\geq 0\}$. To this end we shall prove first boundedness
in probability.

\noindent We fix the initial process $\varphi$ and consider the
solution of \eqref{sde0}.
\begin{proposition}\label{Beschr.Masse}
Under hypotheses ${\bf(H_0)}$,  ${\bf(H_1)}$ and  ${\bf(H_2)}$ the process $(x_t)_{t\geq 0}$
is bounded in probability.
\end{proposition}
\noindent For the proof of the proposition we need some preparation.
Let $\eta\colon [0,\infty)\times\Omega\to \R$ be
 a progressively measurable process with locally square integrable sample paths.
Consider a one-dimensional Brownian motion $(\beta (t) )_{t\ge 0}$ and for $\mu>0$ let us
introduce the following equation
\begin{equation*}
 \left\{
\begin{array}{ll}
dv(t)=-\mu v(t)dt+\eta(t,\omega)d\beta(t),\quad
t\geq0\\
v(0)=0.
\end{array}
\right.
\end{equation*}

\noindent If we denote by $(v_{\mu}(\cdot))$ its solution we have
$$
v_{\mu}(t)=\int_0^t\ee^{-\mu(t-s)}\eta(s,\omega)\:d\beta(s).
$$

\noindent The following lemma gives an estimate for the process
$v_{\mu}(\cdot)$.
\begin{lemma}\label{stoch-fact}
For $2<p<+\infty$ and $\mu>0$, there exists a positive
constant $a_{p,\mu}$ such that
$$\lim\limits_{\mu\to +\infty}a_{p,\mu}= 0$$ and
\begin{equation}\label{estimate}
\E\sup\limits_{0\leq t\leq T}|v_{\mu}(t)|^p\leq a_{p,\mu}\cdot
\E\int_0^T|\eta(s,\omega)|^p\:ds , \mbox{ for every }T>0.
\end{equation}
\end{lemma}
\begin{proof}
Fix $2<p<\infty$, $T>0$ and assume that $\E \int_0^T|\eta(s,\omega)|^p\:ds < \infty$.
Let $\frac 1p <\alpha<\frac 12$ and define
$$
y(t)\defeq
\int_0^t(t-s)^{-\alpha}\ee^{-\mu(t-s)}\eta(s,\omega)\:d\beta(s), \quad
t\ge 0.
$$
Using the factorization formula (see \cite[Sect. 7.1]{DPZ1})
$$
\int_0^t
\ee^{-\mu(t-s)}\eta(s,\omega)\:d\beta(s)=\frac{\sin\pi\alpha}{\pi}R_{\alpha}y(t)
$$
where
$$
R_{\alpha}f(t)=\int_0^t(t-s)^{\alpha-1}\ee^{-\mu(t-s)}f(s)\:ds
$$

\noindent defines a bounded linear operator from $L^p([0,T],\R)$
into $C([0,T],\R)$. Indeed, take a function $f$ in $ L^p([0,T],\R)$,
then we have
\begin{equation*}
\begin{split}
|R_{\alpha}f(t)|&\leq
\int_0^t(t-s)^{\alpha-1}\ee^{-\mu(t-s)}|f(s)|\:ds\\
&\leq \|f\|_{L^p([0,T],\R)}\left(
\int_0^t(t-s)^{(\alpha-1)p/(p-1)}\ee^{-\mu
p(t-s)/(p-1)}\:ds\right)^{\frac{p-1}{p}}\\
&\leq \|f\|_{L^p([0,T],\R)}\left(
\int_0^{+\infty}s^{(\alpha-1)p/(p-1)}\ee^{-\mu
ps/(p-1)}\:ds\right)^{\frac{p-1}{p}}\\
&=\|f\|_{L^p([0,T],\R)}\left(\frac{p-1}{\mu p}\right)^{\alpha-\frac
1p}\Gamma \left( \frac{\alpha p-1}{p-1}\right)^{1-\frac 1p}.
\end{split}
\end{equation*}

\noindent Therefore
\begin{equation*}
\begin{split}
\E\left(\sup\limits_{0\leq t\leq
T}\left|\int_0^T\ee^{-\mu(t-s)}\eta(s,\omega)\:d\beta(s)\right|^p\right)^{\frac
1p}&=\E\left(\sup\limits_{0\leq t\leq
T}\left|\frac{\sin\pi\alpha}{\pi}R_{\alpha}y(t)\right|^p\right)^{\frac 1p}\\
&\leq
\|R_{\alpha}\|\left(\E\|y(\cdot)\|_{L^p([0,T],\R)}^p\right)^{\frac
1p}\\ &\leq \left(\frac{p-1}{\mu p}\right)^{\alpha-\frac 1p}\Gamma
\left( \frac{\alpha p-1}{p-1}\right)^{1-\frac 1p}
\left(\E\|y(\cdot)\|_{L^p([0,T],\R)}^p\right)^{\frac 1p}.
\end{split}
\end{equation*}

\noindent Using Burkholder-Davis-Gundy's inequality we obtain
\begin{equation*}
\begin{split}
\E\|y\|_{L^p([0,T],\R)}^p&=\E\int_0^T|y(t)|^p\:dt\\
&=\int_0^T\E|\int_0^t(t-s)^{-\alpha}\ee^{-\mu(t-s)}\eta(s,\omega)\:d\beta(s)|^p\:dt\\
&\leq
c_p\E\int_0^T\left(\int_0^t(t-s)^{-2\alpha}|\ee^{-\mu(t-s)}\eta(s,\omega)|^2\:ds\right)^{\frac
p2}\:dt\\
(\mbox{Young's inequality} )&\leq
c_p\left(\int_0^Ts^{-2\alpha}\ee^{-2\mu s}\:ds\right)^{\frac
p2}\cdot \E\int_0^T|\eta(s,\omega)|^p\:ds\\
&\leq c_p \left(\frac{1}{2\mu}
\int_0^{+\infty}\left(\frac{t}{2\mu}\right)^{-2\alpha}e^{-t}\:dt\right)^{\frac
p2}\cdot\E\int_0^T|\eta(s,\omega)|^p\:ds.
\end{split}
\end{equation*}
Hence we have
\begin{equation*}
\begin{split} \E\|y\|^p_{L^p([0,T],\R)}&\leq
c_p\cdot\left(\frac{1}{(2\mu)^{1-2\alpha}}\Gamma(1-2\alpha)\right)^{\frac p2}\E\int_0^T|\eta(s,\omega)|^p\:ds\\
&= c_{p,\mu}\cdot \E\int_0^T|\eta(s,\omega)|^p\:ds,
\end{split}
\end{equation*}
where $c_{p,\mu}\defeq c_p
\left(\frac{1}{(2\mu)^{1-2\alpha}}\Gamma(1-2\alpha)\right)^{\frac
p2} $. Therefore we deduce
\begin{equation*}
\E\left(\sup\limits_{0\leq t\leq
T}\left|\int_0^t\ee^{-\mu(t-s)}\eta(s,\omega)\:d\beta(s)\right|^p\right)\leq
a_{p,\mu}\E\int_0^T|\eta(s,\omega)|^p\:ds,
\end{equation*}

\noindent where $$ a_{p,\mu}\defeq c_{p,\mu}\cdot
\left(\frac{p-1}{\mu p}\right)^{p\alpha-1}\Gamma \left( \frac{\alpha
p-1}{p-1}\right)^{p-1}.$$
\end{proof}
\noindent We are now in the position to complete the proof of the
proposition.
\begin{proof}
\noindent Let $\lambda \ge 1$. For $x\in \R^d$ we define
$$
R_{\lambda}(x)\defeq 2\la f(x),x\ra+\lambda |x|^2.
$$

\noindent  By hypothesis ${\bf (H_0)}$ there exists $\A>0$ such that
$$
\frac{\la f(x),x\ra}{|x|^2}\leq -\lambda,\quad |x|\geq \A.
$$

\noindent Again by ${\bf (H_0)}$ we can find $B\geq 0$ independent
of $\lambda$ such that
\begin{equation}\label{R-estimate}
R_{\lambda}(x)\leq B+\lambda \A^2\quad\mbox{for all $x\in \R^d$} .
\end{equation}

\noindent We now consider the solution $x(\cdot)$ of equation
\eqref{sde0} and  set $z(t)\defeq |x(t)|^2$, $t\geq 0$. Then
It\^{o}'s formula implies that for fixed $t\geq 0$ we have
\begin{equation}
\begin{split}
dz(t)&=2\la f(x(t)),x(t)\ra dt+2\la g(x_t),x(t)\ra dt+\Norm{h(x_t)}^2
dt+2\la x(t),h(x_t)dB(t)\ra\\
&=\Big( -\lambda z(t)+R_{\lambda}(x(t))+2\la g(x_t),x(t)\ra
dt+\Norm{h(x_t)}^2\Big) dt+2\la x(t),h(x_t)dB(t)\ra\\
&\leq\Big( -\lambda z(t)+R_{\lambda}(x(t))+2\la
g(x_t)-g(0),x(t)\ra+2\la
g(0),x(t)\ra\\&\qquad+2\Norm{h(x_t)-h(0)}^2+2\Norm{h(0)}^2\Big)dt+2\la x(t),h(x_t)dB(t)\ra\\
&\leq \Big( -\lambda z(t)+R_{\lambda}(x(t))+2L\|x_t\|^2+2\la
g(0),x(t)\ra+2\|h(0)\|^2\Big)dt+2\la x(t),h(x_t)dB(t)\ra\\
&\leq   \Big(-\lambda z(t)+R_{\lambda}(x(t))+3L\|x_t\|^2+\frac
1L|g(0)|^2+2\|h(0)\|^2\Big)dt+2\la x(t),h(x_t)dB(t)\ra,
\end{split}
\end{equation}

\noindent where we used the estimate $$\la g(0),x(t)\ra\leq
\frac{L}{2}|x(t)|^2+\frac{1}{2L}|g(0)|^2\leq
\frac{L}{2}\|x_t\|^2+\frac{1}{2L}|g(0)|^2.$$

\noindent

\noindent Set $D\defeq \frac 1L|g(0)|^2+2\|h(0)\|^2 $, so the
variation of constants formula yields
\begin{equation*}
\begin{split}
z(t)&\leq z(0)\ee^{-\lambda t}+\int_0^t
\ee^{-\lambda(t-s)}\Big(R_{\lambda}(x(s))+3L\|x_s\|^2+D\Big)\:ds+2\int_0^t\ee^{-\lambda(t-s)}\la
x(s),h(x_s)dB(s)\ra\\
& \leq z(0)\ee^{-\lambda
t}+\A^2+\frac{B+D}{\lambda}+\frac{3L}{\lambda}\sup\limits_{-r\leq
s\leq t}|x(s)|^2+2\int_0^t\ee^{-\lambda(t-s)}\la x(s),h(x_s)dB(s)\ra.
\end{split}
\end{equation*}

\noindent There exists a one-dimensional Brownian motion $\beta$ with respect to the same filtration
such that
\begin{equation*}
\la x(s),h(x_s)dB(s)\ra = \eta(s,\omega)\,d\beta(s),
\end{equation*}
where
\begin{equation*}
\eta(s,\omega)= \Big(\sum_{j=1}^m \big(\sum_{i=1}^d x_i(s) h_{ij}(x_s)\big)^2\Big)^{1/2}.
\end{equation*}
By ${\bf (H_2)}$, we get
\begin{equation}\label{eta}
|\eta(s,\omega)|^3 \le |x(s)|^3\Norm{h(x_s)}^3 \le 4 \: |x(s)|^3 \big( L^{3/2}\|x_s\|^3 + D^{3/2} \big).
\end{equation}

\noindent Hence for $0 \le t\leq r$ we obtain
\begin{equation*}
\ee^tz(t)\leq
z(0)+\ee^r\Big(\A^2+\frac{B+D}{\lambda}\Big)+\frac{3L}{\lambda}\ee^r\sup\limits_{-r\leq
s\leq t}|x(s)|^2+2\ee^{r}\sup\limits_{0\leq t\leq
r}\Big|\int_0^t\ee^{-\lambda (t-s)} \eta(s)\, d\beta(s)\Big|.
\end{equation*}

\noindent Now using Lemma \ref{stoch-fact} and \eqref{eta} we get
\begin{equation*}
\begin{split}
\E\sup\limits_{0\leq t\leq r}\Big|\int_0^t \ee^{-\lambda(t-s)}
\eta(s) d\beta(s)\Big|^3&\leq a_{3,\lambda}\: r \: \E\|\eta_r\|^3\\
&\leq 4\:a_{3,\lambda}\: r \: \Big( D^{3/2} \E\|x_r\|^3 + L^{3/2} \big(\E\|x_r\|^6 + \E(\|\varphi\|^3\|x_r\|^3)\big)\Big)\\
&\leq 2\:a_{3,\lambda}\: r \Big(D^{3/2} \big(\E\|x_r\|^6 + 1\big) +  L^{3/2} \big(3\,\E\|x_r\|^6 + \E\|\varphi\|^6\big)\Big).
\end{split}
\end{equation*}

%

\noindent If we choose $\kappa\in (1,\ee^{3r})$ and
$\gamma>1$ such that $(a+b+c+d)^3\leq \kappa
a^3+\gamma (b^3+ c^3+ d^3)$ for all $a,\:b,\:c,\:d\geq 0$
we have

\begin{equation*}
\begin{split}
\E\sup\limits_{0\leq t\leq r}|\ee^tz(t)|^3&\leq \kappa
\E|z(0)|^3+\gamma
\ee^{3r}\Big(\A^2+\frac{B+D}{\lambda}\Big)^3+\gamma\frac{27L^3}{\lambda^3}\ee^{3r}\big(\E \|\varphi\|^6 + \E\sup\limits_{0\leq
s\leq r}|\ee^s z(s)|^3\big)\\&+\gamma 16\:a_{3,\lambda}\: r\:\ee^{3r}
\Big( D^{3/2} \big(\E\sup\limits_{0\leq t\leq r}|\ee^{t}z(t)|^3+1\big)
+ L^{3/2} \big(3\E\sup\limits_{0\leq t\leq r}|\ee^{t}z(t)|^3 + \E\|\varphi\|^6\big)\Big).
\end{split}
\end{equation*}

\noindent Let $\psi(s):=|\varphi(s)|^2$, $s \in [-r,0]$. We define the function $V:\:C([-r,0],\R)\rightarrow
\R^+$ by  $$V(\zeta)\defeq \sup\limits_{-r\leq s\leq
0}(\ee^{3s}|\zeta(s)|^3). $$

\noindent We deduce from the above calculation that
\begin{equation}
\begin{split}
\E V(z_r)&\leq \kappa \ee^{-3r} \E V(\psi)+\gamma
\Big(\A^2+\frac{B+D}{\lambda}\Big)^3+\gamma\frac{27L^3}{\lambda^3}\ee^{3r}\big(\E
V(\psi)+\E V(z_r)\big)\\&+16\,\gamma\:a_{3,\lambda}\: r \Big(\E V(z_r) \ee^{3r}\big( D^{3/2}+3\,L^{3/2}\big) +  \ee^{3r}L^{3/2}\E V(\psi)+
D^{3/2}\Big).
\end{split}
\end{equation}

\noindent Hence, for $\lambda_{\ast}$ sufficiently large, we get
\begin{equation}\label{V-estimate}
\begin{split}
\E V(z_r)&\leq \delta \E V(\psi)+\rho,
\end{split}
\end{equation}
where
\begin{eqnarray*}
\delta&\defeq&\frac{\kappa
\ee^{-3r}+\gamma\frac{27L^3}{\lambda_{\ast}^3}\ee^{3r}+16\gamma\:a_{3,\lambda_{\ast}}\:
r\ee^{3r}L^{3/2}}{1-\gamma \ee^{3r}\big(\frac{27L^3}{\lambda_{\ast}^3}+16\:a_{3,\lambda_{\ast}}\:
r\,(D^{3/2}+3L^{3/2})\big)}<1,\\
\rho&\defeq&\frac{\gamma
\Big(A_{\lambda_{\ast}}^2+\frac{B+D}{\lambda_{\ast}}\Big)^3+16\gamma\:a_{3,\lambda_{\ast}}\:
r\,D^{3/2}}{1-\gamma \ee^{3r}\big(\frac{27L^3}{\lambda_{\ast}^3}+16\:a_{3,\lambda_{\ast}}\:
r\,(D^{3/2}+3L^{3/2})\big)},
\end{eqnarray*}
provided that $\E V(z_r)<\infty$ ($\E V(\psi)$ is finite by assumption). To see that this property holds,
apply the previous calculation to the process $|x(t)|$ stopped as soon as it reaches level $N$ and then
let $N \to \infty$.
Iterating \eqref{V-estimate} we get
\begin{equation}\label{Iteration}
\E V(z_{kr})\leq \delta^k\E V(\psi)+\frac{\rho}{1-\delta}\leq \E
V(\psi)+\frac{\rho}{1-\delta},\quad\mbox{for all $k\in\NN$}.
\end{equation}

\noindent Let $t\geq 0$. Then there exists $k\in\NN_0$ such that
$kr\leq t\leq (k+1)r$ and we have
\begin{equation} \label{spliting}
\E\|z_t\|^3\leq \E\|z_{kr}\|^3+\E\|z_{(k+1)r}\|^3.
\end{equation}

\noindent Using \eqref{Iteration} we obtain
$$
\E\|z_{kr}\|^3=\E\sup\limits_{-r\leq s\leq 0}|z_{kr}(s)|^3\leq
\ee^{3r}\E V(z_{kr})\leq \ee^{3r} \Big(\E
V(\psi)+\frac{\rho}{1-\delta}\Big).
$$

\noindent Combining this with \eqref{spliting} yields
\begin{equation}\label{moment-estimate}
\sup\limits_{t\geq 0}\E\|x_t\|^6<+\infty.
\end{equation}

\noindent This implies the boundedness in probability of the segment
process $(x_t)_{t\geq 0}$ and the proposition is proved.
\end{proof}

\noindent The following theorem is our main result in this section.

\begin{theorem}
Under hypotheses ${\bf(H_0)}$, ${\bf(H_1)}$,  ${\bf(H_2)}$ the
family $\{\mathcal{L}(x_t),\; t\geq 0\}$ is tight.
\end{theorem}

\begin{proof}
 From
\eqref{moment-estimate} we have in particular the boundedness in
probability of the finite dimensional process $(x(t))_{t\geq 0}$ and
hence the family $\{\mathcal{L}(x(t)),\; t\geq 0\}$ is tight. To
prove the theorem, it is sufficient to show that
\begin{equation}\label{hoelder}
\lim\limits_{\delta\to 0}\sup\limits_{t\geq 0}\PP\left(
\sup\limits_{{\stackrel{t\leq u\leq v\leq t+r}{v-u\leq
\delta}}}|x(v)-x(u)|\geq \gamma\right)=0\quad \mbox{for any
$\gamma>0$}.
\end{equation}

\noindent To shorten notation let
$$
\g(\eta)\defeq g(\eta)+f(\eta(0)),\quad \eta\in C([-r,0],\R^d).
$$
\noindent Thus we can write

$$
x(t)=x(0)+\int_0^t \g(x_s)\,\mathrm{d}s + \int_0^t
h(x_s)\,\mathrm{d}B(s)
$$

\noindent and we have
\begin{equation}
\begin{split}
\PP\left( \sup\limits_{{\stackrel{t\leq u\leq v\leq t+r}{v-u\leq
\delta}}}|x(v)-x(u)|\geq \gamma\right)&\leq \PP\left(
\sup\limits_{{\stackrel{t\leq u\leq v\leq t+r}{v-u\leq
\delta}}}\int_u^v|\g(x_s)|\:ds\geq \frac\gamma 2\right)\\&+\PP\left(
\sup\limits_{{\stackrel{t\leq u\leq v\leq t+r}{v-u\leq
\delta}}}\Big|\int_u^v h(x_s)\:dB(s)\Big|\geq \frac\gamma 2\right)\\
&=M_t+N_t.
\end{split}
\end{equation}
\noindent Let $\ve$, $R>0$. For the term $M_t$ we have
\begin{equation*}
\begin{split}
M_t\leq  \PP&\left( \sup\limits_{{\stackrel{t\leq u\leq v\leq
t+r}{v-u\leq \delta}}}\int_u^v|\g(x_s)|\:ds\geq \frac\gamma 2
\:\Big|\:\|x_t\|\leq R,\:\|x_{t+r}\|\leq R \right)\\&+\PP\Big(\{
\|x_t\|>R\}\Big)+\PP\Big(\{ \|x_{t+r}\|>R\}\Big).
\end{split}
\end{equation*}

\noindent Since the process $(x_t)_{t\geq 0}$ is bounded in
probability we can choose $R$ so large such that
$$
\PP\Big(\{ \|x_t\|>R\}\Big)+\PP\Big(\{ \|x_{t+r}\|>R\}\Big)\leq
\frac\ve 2\quad \mbox{for all $t\geq 0$}.
$$

\noindent By ${\bf(H_1)}$, $\g(x_s)$, $s\in [t-r,t+r]$ is bounded
on the set $\{\|x_t\|\leq R\}\cap\{\|x_{t+r}\|\leq R\}$, so it follows
that there exists $\delta_0>0$ such that
$$
\PP\left( \sup\limits_{{\stackrel{t\leq u\leq v\leq t+r}{v-u\leq
\delta}}}\int_u^v|\g(x_s)|\:ds\geq \frac\gamma 2
\:\Big|\:\|x_t\|\leq R,\:\|x_{t+r}\|\leq R \right)=0\quad\mbox{for
any $\delta<\delta_0$}.
$$

\noindent Therefore we get
$$
\lim\limits_{\delta\to 0} \sup\limits_{t\ge 0} M_t=0.
$$

\noindent For the term $N_t$ we define
$$
J(t)\defeq \int_0^t h(x_s)\:dB(s).
$$

\noindent Using Burkholder's inequality and ${\bf(H_2)}$, we get
\begin{equation*}
\begin{split}
\E|J(t)-J(s)|^{6}&=\E\Big|\int_{s}^{t}h(x_u)\:dB(u)\Big|^{6}\\
&\leq c\, \E\Big(\int_{s}^{t}\Norm{h(x_u)}^2\,du\Big)^{3}\\
\mbox{(Jensen's inequality)}\quad &\leq \bar c\, |t-s|^{3} \big(\sup\limits_{u\geq 0}\E\|x_u\|^{6}+1\big),
\end{split}
\end{equation*}
where $\bar c$ depends on $L$ and $D$.
Using \eqref{moment-estimate} and Kolmogorov's
tightness criterion (see \cite[2.4.11]{KS} or \cite{S09}) we infer that
$$
\lim\limits_{\delta\to 0}\sup\limits_{t\ge 0}N_t=
\lim\limits_{\delta\to 0}\sup\limits_{t\ge 0}\PP\left(
\sup\limits_{{\stackrel{t\leq u\leq v\leq t+r}{v-u\leq
\delta}}}\Big|\int_u^v h(x_s)\:dB(s)\Big|\geq \frac\gamma
2\right)=0.
$$

\noindent This establishes \eqref{hoelder} and the proof is
complete.
\end{proof}

\section{invariant measures}
\noindent In this section we discuss the existence of an invariant
measure $\mu$ for the segment process $(x_t)_{t\geq 0}$. Since in
the last section we proved tightness of this process, in order to
apply Krylov-Bogoliubov's theorem we need to prove the Feller
property of $(x_t)_{t\geq 0}$.
\begin{proposition}
Assume hypotheses ${\bf(H_0)}$,  ${\bf(H_1)}$ and  ${\bf(H_2)}$.
Let $(\varphi_m)_{m\in\NN}$ be a sequence in $C([-r,0],\R^d)$ such
that $\varphi_m\xrightarrow[m\to +\infty]{\|\cdot\|} \varphi$. Let
$x^m$ (resp. $x$) be the solutions to \eqref{sde0} with initial
condition $\varphi_m$ (resp. $\varphi$). Then for any $t>0$,
\begin{equation}
\E\sup\limits_{t-r\leq s\leq t}|x^m(s)-x(s)|^4\rightarrow 0\quad
\mbox{as}\:\:m\to+\infty.
\end{equation}

\noindent In particular, $(x_t)_{t\geq 0}$ is a Feller process.
\end{proposition}
\begin{proof}
Using It\^{o}'s formula we can write
\begin{equation}
\begin{split}
d|x^m(t)-x(t)|^2&=2\la
f(x^m(t))-f(x(t))+g(x^m_t)-g(x_t),x^m(t)-x(t)\ra
dt\\&+\Norm{h(x_t^m)-h(x_t)}^2dt+dM(t),
\end{split}
\end{equation}

\noindent where $$ M(t)\defeq \int_0^t 2\la
x^m(s)-x(s),(h(x_s^m)-h(x_s))\:dB(s)\ra$$ is a martingale with
quadratic variation process bounded by
$4L\int_0^t\|x^m_s-x_s\|^4\:ds$. Thus if we define
$M^{\ast}(t)=\sup\limits_{s\leq t} M(s)$ we obtain
$$
\|x^m_t-x_t\|^2\leq
\|\varphi_m-\varphi\|^2+L\int_0^t\|x^m_s-x_s\|^2\:ds+M^{\ast}(t).
$$

\noindent This implies
\begin{equation}
\begin{split}
\E \|x^m_t-x_t\|^4&\leq
3\E\left(\|\varphi_m-\varphi\|^4+L^2\Big(\int_0^t\|x^m_s-x_s\|^2\:ds\Big)^2+\Big(M^{\ast}(t)\Big)^2\right)\\
&\leq 3
\left(\|\varphi_m-\varphi\|^4+L^2t\int_0^t\E\|x^m_s-x_s\|^4\:ds+4L\int_0^t\E\|x^m_t-x_t\|^4\:ds\right).
\end{split}
\end{equation}

\noindent Hence, by Gronwall's inequality, we obtain
$$
\E \|x^m_t-x_t\|^4\leq 3\|\varphi_m-\varphi\|^4\ee^{12Lt+3L^2t^2}.
$$

\noindent This implies in particular that for
$\psi:\:C([-r,0],\R^d)\rightarrow\R^d$ bounded and continuous we
have
$$
\lim\limits_{m\to+\infty}\E\psi((x_m)_t)=\E\psi(x_t)\quad\mbox{for
any $t>0$},
$$
which yields the Feller property.
\end{proof}

\noindent Now, by the Krylov-Bogoliubov Theorem (see Sect.3.1 in
\cite{DaZa:96}) we have the following result.
\begin{theorem}
Under hypotheses ${\bf(H_0)}$,  ${\bf(H_1)}$ and  ${\bf(H_2)}$ the
segment process $(x_t)_{t\geq 0}$ corresponding to \eqref{sde0} has
an invariant measure.
\end{theorem}
\begin{rem}
Our proofs show that hypothesis ${\bf (H_0)} $ can be weakened by
requiring that\\ $\limsup\limits_{|v|\to +\infty}\frac{\la
f(v),v\ra}{|v|^2}<-\lambda$, for $\lambda$ sufficiently large
positive constant (which depends on $L,\,r$ and $h(0)$).
\end{rem}


\begin{thebibliography}{99}

\bibitem{bakhtin} Y.\ Bakhtin and J.C.\ Mattingly, {\it Stationary solutions of stochastic differential equations with memory and stochastic partial differential equations}, Commun.\ Contemp.\ Math \textbf{7} (5) (2005), 553--582.

\bibitem{bonaccorsi} S.\ Bonaccorsi and G.\ Tessitore, {\it Asymptotic behavior of infinite dimensional stochastic differential equations by anticipative variation of constants formula}, Appl.\ Math.\ Optim \textbf{44} (2001), 203--225.



\bibitem{chow} P.-L.\ Chow and R.Z.\ Khasminskii, {\it Stationary solutions of nonlinear stochastic evolution equations}, Stoch.\ Anal.\ Appl \textbf{15} (5) (1997), 671--699.




\bibitem{DPZ1} G.\ Da\ Prato and J.\ Zabczyk,
\emph{Stochastic Equations in Infinite Dimensions}, Encyclopedia of
mathematics and its applications, Vol.\ 45, Cambridge University
Press, Cambridge, 1992.

\bibitem{DaZa:96}
G.~Da Prato and J.~Zabczyk,
\newblock \textit{{E}rgodicity for {I}nfinite {D}imensional   {S}ystems},
\newblock London Mathematical Society Lecture Notes, vol.~229, Cambridge University Press, 1996.


\bibitem{HMS} M.\ Hairer, J.C.M.\ Mattingly and M.\ Scheutzow, {\it Asymptotic coupling and a weak form of Harris' theorem with applications to stochastic delay equations}, submitted.

\bibitem{KS}
I.~Karatzas and S.~Shreve,
\newblock \textit{Brownian Motion and Stochastic Calculus},
\newblock Graduate Texts in Mathematics, vol.~113, Springer Verlag, New York, 1988.



\bibitem{reiss} M.\ Rei\ss, M.\ Riedle and O.\ van Gaans, {\it Delay differential equations driven by L\'evy processes: Stationarity and Feller properties}, Stoch.\ Process.\ Appl \textbf{116} (2006), 1409--1432.

\bibitem{scheutzow} M.K.R.\ Scheutzow, {\it Qualitative behaviour of stochastic delay equations with a bounded memory}, Stochastics \textbf{12} (1984), 41--80.

\bibitem{RS} M.K.R.\ Scheutzow and M.K.\ von Renesse, {\it Existence and uniqueness of solutions of stochastic functional differential equations}, submitted.

\bibitem{S09} M.K.R.\ Scheutzow, {\it Chaining techniques and their application to stochastic flows}, in: Trends in {S}tochastic {A}nalysis,
eds: J.\ Blath, P.\ M\"orters, M.\ Scheutzow,
LMS Lecture Notes Series 353, Cambridge University Press, Cambridge, 2009, to appear.

\bibitem{vangaans} O.\ van Gaans and S.\ Verduyn\ Lunel, {\it Long term behavior of dichotomous stochastic differential equations in Hilbert spaces}, Commun.\ Contemp.\ Math \textbf{6} (3) (2004), 349--376.

\end{thebibliography}
\end{document}